\documentclass[11pt,leqno]{article}
\usepackage{exscale,relsize}
\usepackage{amsmath}
\usepackage{amsfonts}
\usepackage{amssymb}
\usepackage{calc}
\usepackage{theorem}
\usepackage{pifont}      
\usepackage{color}
\definecolor{labelkey}{rgb}{0,0.08,0.45}
\definecolor{refkey}{rgb}{0,0.6,0.0}
\definecolor{Brown}{rgb}{0.45,0.0,0.05}
\definecolor{lime}{rgb}{0.00,0.95,0.0}
\definecolor{lblue}{rgb}{0.5,0.5,0.99}
\usepackage{soul}
\newcommand{\scal}[2]{\langle{{#1},{#2}}\rangle}

\newcommand{\RR}{\ensuremath{\mathbb R}}

\newcommand{\RX}{\ensuremath{\,\left]-\infty,+\infty\right]}}

\newcommand{\thalb}{\ensuremath{\tfrac{1}{2}}}

\newcommand{\menge}[2]{\big\{{#1} \mid {#2}\big\}}

\newcommand{\spand}{\operatorname{span}}

\newcommand{\dom}{\ensuremath{\operatorname{dom}}}

\newcommand{\gra}{\ensuremath{\operatorname{gra}}}

\newcommand{\ran}{\ensuremath{\operatorname{ran}}}

\newcommand{\Id}{\ensuremath{\operatorname{Id}}}

\renewcommand{\phi}{\ensuremath{\varphi}}

\newcommand{\To}{\ensuremath{\rightrightarrows}}

\newtheorem{theorem}{Theorem}[section]
\newtheorem{lemma}[theorem]{Lemma}
\newtheorem{fact}[theorem]{Fact}
\newtheorem{corollary}[theorem]{Corollary}
\newtheorem{proposition}[theorem]{Proposition}
\newtheorem{definition}[theorem]{Definition}

\theoremstyle{plain}{\theorembodyfont{\rmfamily}
}
\theoremstyle{plain}{\theorembodyfont{\rmfamily}
}
\theoremstyle{plain}{\theorembodyfont{\rmfamily}
}
\theoremstyle{plain}{\theorembodyfont{\rmfamily}
\newtheorem{example}[theorem]{Example}}
\theoremstyle{plain}{\theorembodyfont{\rmfamily}
\newtheorem{remark}[theorem]{Remark}}
\theoremstyle{plain}{\theorembodyfont{\rmfamily}
}

\begin{document}

\sethlcolor{lime}
\title{{\sffamily {Examples of discontinuous 
maximal monotone linear operators }\\  
{and the solution to a recent problem posed by B.F.~Svaiter}}}

\author{
Heinz H.\ Bauschke\thanks{Mathematics, Irving K.\ Barber School,
UBC Okanagan, Kelowna, British Columbia V1V 1V7, Canada. E-mail:
\texttt{heinz.bauschke@ubc.ca}.},  Xianfu
Wang\thanks{Mathematics, Irving K.\ Barber School, UBC Okanagan,
Kelowna, British Columbia V1V 1V7, Canada. E-mail:
\texttt{shawn.wang@ubc.ca}.}, and Liangjin\
Yao\thanks{Mathematics, Irving K.\ Barber School, UBC Okanagan,
Kelowna, British Columbia V1V 1V7, Canada.
E-mail:  \texttt{ljinyao@interchange.ubc.ca}.}. }

\date{September 14, 2009}

\maketitle

\begin{abstract} \noindent
In this paper, we give two explicit examples of unbounded linear
maximal monotone operators. 
The first unbounded linear maximal monotone operator  $S$ on $\ell^{2}$ is skew.
We show  its domain is a proper subset of the domain of its adjoint $S^*$,
and  $-S^*$ is not maximal monotone. 
This gives a negative answer to a recent question posed by Svaiter.
The second unbounded linear maximal monotone operator is the inverse Volterra operator $T$ on $L^{2}[0,1]$.
We compare the domain of $T$ with the domain of its adjoint $T^*$ and show that the skew part of $T$ admits two distinct
linear maximal monotone skew extensions.
 These unbounded linear maximal monotone operators show that the constraint qualification
 for the maximality of the sum of maximal monotone operators can not be significantly weakened,
 and they are simpler than the example given by Phelps-Simons. Interesting consequences on
Fitzpatrick functions for sums of two maximal monotone operators are also given.
\end{abstract}



\section{Introduction}
Linear monotone operators play important roles in modern monotone operator
theory \cite{BauBorPJM,BBW,PheSim,Si2,Voisei06,VZ08}, and they are
examples that delineate the boundary of the general theory.
In this paper, we explicitly construct two unbounded linear monotone operators (not full domain, linear and single-valued on their domains). They answer one of Svaiter's question, have some
interesting consequences
on Fitzpatrick functions for sums of two maximal monotone operators, and show that the constraint qualification for the maximality of the sum of
maximal monotone operators can not be weaken
significantly, see \cite{rocka}, \cite[Theorem~5.5]{SiZ} and \cite{Voisei06}. Our examples are simpler than the one
given by \cite{PheSim}.

The paper is organized as follows.
Basic facts and auxiliary results are recorded in Section~\ref{basic}.
In Section~\ref{s:skew}, we construct an unbounded maximal monotone skew operator $S$ on $\ell^2$.
For a maximal monotone skew operator,
it is well known that its domain is always a subset of the domain of its adjoint.
An interesting question remained is  whether or not both of the domains are always same.
The maximal monotone skew operator $S$ enjoys the property that the domain of $-S$ is
a proper subset of the domain of its adjoint $S^*$, see Theorem~\ref{PE:01a}.
Svaiter asked in \cite{SV}
whether or not $-S^*$ (termed $S^{\vdash}$ in \cite{SV}) 
is maximal monotone provided that $S$ is maximal skew.
This operator also answers Svaiter's
question in the negative, see Theorem~\ref{Pl:1}.
In Section~\ref{s:vol} we systematically study the inverse Volterra operator $T$.
We show that $T$ is neither skew nor symmetric and compare the domain of $T$ with the domain of its adjoint $T^*$.
It turns out that the skew part of $T$: $S=\tfrac{T-T^*}{2}$ admits two distinct linear maximal monotone and skew extensions
even the the domain of $S$ is a dense linear subspace in $L^{2}[0,1]$.
It was shown that
Fitzpatrick functions $F_{A+B}=F_{A}\Box_{2}F_{B}$ when $A,B$ are maximal monotone linear relations and
$\dom A-\dom B$ is a closed subspace, see \cite[Theorem 5.10]{BWY3}. Using these unbounded linear maximal monotone operators in Sections~\ref{s:skew}
and \ref{s:vol} we also show that
the constraint qualification $\dom A-\dom B$ being closed can not be significantly weakened either.

Throughout this paper, we assume that
 \begin{equation*}\text{$X$ is a real Hilbert space,
 with inner product $\langle\cdot,\cdot\rangle$.}
\end{equation*}
Let $S$ be a set-valued operator (also known as multifunction)
from $X$ to $X$.
 We say that $S$ is \emph{monotone} if
\begin{equation*}
\big(\forall (x,x^*)\in \gra S\big)\big(\forall (y,y^*)\in\gra
S\big) \quad \scal{x-y}{x^*-y^*}\geq 0,
\end{equation*}
where
$\gra S := \menge{(x,x^*)\in X\times X}{x^*\in Sx}$;
$S$ is said to be \emph{maximal monotone} if no proper enlargement
(in the sense of graph inclusion) of $S$ is monotone. We say $T$ is
 a maximal monotone extension
of $S$ if $T$ is maximal monotone and $\gra T\supseteq\gra S$.
The \emph{domain} of $S$ is
$\dom S := \{x\in X\mid Sx\neq\varnothing\}$, and its \emph{range} is $\ran
S: = S(X) = \bigcup_{x\in X} Sx$.

We say $S$ is a \emph{linear relation} if $\gra S$ is linear.
 The \emph{adjoint} of $S$, written $S^*$, is defined by
\begin{equation*}
\gra S^* :=
\menge{(x,x^*)\in X\times X}{(x^*,-x)\in (\gra S)^\bot},
\end{equation*}where, for any subset $C$ of a Hilbert space $Z$,
$C^\bot := \menge{z\in Z}{z|_C \equiv 0}$.
 We say a linear relation $S$ is
\emph{skew} if $\langle x,x^*\rangle=0,\; \forall (x,x^*)\in\gra S$,
and
$S$ is a \emph{maximal monotone skew operator} if $S$ is a maximal
monotone operator and $ S$ is skew.
Svaiter introduced $S^{\vdash}$ in \cite{SV}, which is defined by
\begin{align*}
\gra S^{\vdash}:=\menge{(x,x^*)\in X\times X}{(x^*,x)\in (\gra S)^\bot}.\end{align*}
Hence  $S^{\vdash}=-S^*$.
For each function $f:X\rightarrow \RX$, $f^*$ stands for the 
\emph{Fenchel conjugate} given by
$$f^{*}(x^*)=\sup_{x\in X}\big(\scal{x^*}{x}-f(x)\big)\quad \forall x^*\in X.$$

\section{Auxiliary results and facts}\label{basic}
In this section we gather some facts about linear relations, 
monotone operators, and Fitzpatrick functions.
They will be used frequently in sequel.

\begin{fact}[Cross]\label{Rea:1}
Let $S:X \rightrightarrows X$ be a linear relation.
Then the following hold. 
\begin{enumerate}
\item \label{Th:29} $(S^*)^{-1}=(S^{-1})^*$.
\item\label{Th:31} If $\gra S$ is closed, then $S^{**} = S$.
\item \label{MPP:4} 
If $k\in\RR\smallsetminus\{0\}$, then 
$(kS)^*=kS^*$. 
\item \label{Sia:2b}
$(\forall x\in \dom S^*)(\forall y\in\dom S)$
$\scal{S^*x}{y}=\scal{x}{Sy}$ is a singleton.
\end{enumerate}
\end{fact}
\begin{proof}
\ref{Th:29}: See \cite[Proposition~III.1.3(b)]{Cross}.
\ref{Th:31}: See \cite[Exercise~VIII.1.12]{Cross}.
\ref{MPP:4}: See \cite[Proposition III.1.3(c)]{Cross}.
\ref{Sia:2b}: See \cite[Proposition~III.1.2]{Cross}.
\end{proof}

{If $S\colon X\To X$ is a linear relation that is 
at most single-valued, then we will identify
$S$ with the corresponding linear operator from $\dom S$ to $X$ and
(abusing notation slightly) also write
$S\colon \dom S\to X$. An analogous comment applies
conversely 
to a linear single-valued operator $S$ with domain $\dom S$, which we
will identify with the corresponding at most single-valued linear relation from $X$ to $X$.}

\begin{fact}[Phelps-Simons] \emph{(See \cite[Theorem~2.5 and Lemma~4.4]{PheSim}.)}
\label{EL:11}
Let $S: \dom S\rightarrow X$ be monotone and linear. The following hold.
\begin{enumerate}
\item
If $S$ is maximal monotone, then $\dom S$ is dense 
{(and hence $S^*$ is at
most single-valued).}
\item Assume that $S$ is a skew operator such that $\dom S$ is dense.
 Then $\dom S\subseteq\dom S^*$
and $S^*|_{\dom S}=-S$.
\end{enumerate}
 \end{fact}
\begin{fact}[Br\'ezis-Browder] \emph{(See \cite[Theorem~2]{Brezis-Browder}.)}
\label{Sv:7}
Let $S\colon X \To X$ be a monotone linear relation
such that $\gra S$ is closed. Then the following are equivalent.
\begin{enumerate}
\item
$S$ is maximal monotone.
\item
$S^*$ is maximal monotone.
\item
$S^*$ is monotone.
\end{enumerate}
\end{fact}

For $A\colon X\To X$, the \emph{Fitzpatrick function} associated with $A$
is defined by
\begin{equation}
F_A\colon X\times X\to\RX\colon
(x,x^*)\mapsto \sup_{(a,a^*)\in\gra A}
\big(\scal{x}{a^*}+\scal{a}{x^*}-\scal{a}{a^*}\big).
\end{equation}

Following Penot \cite{Penot2},
if $F\colon X\times X\to\RX$, we set
\begin{equation}
F^\intercal\colon X\times X\colon (x^*,x)\mapsto F(x,x^*).
\end{equation}

\begin{fact}[Fitzpatrick]
\emph{(See {\cite{Fitz88}}.)}
\label{f:Fitz}
Let $A: X\To X$ be monotone. Then $F_A=\langle \cdot,\cdot\rangle$
 on $\gra A$ and  $F_{A^{-1}}= F^{\intercal}_A$.
If $A$ is maximal monotone and $(x,x^*)\in X\times X$, then
 $$F_{A}(x,x^*)\geq \scal{x^*}{x},$$
with equality if and only if $(x,x^*)\in\gra A$.
 \end{fact}

If $A \colon X\to X$ is a linear operator, we write
\begin{equation}
A_+ = \thalb A + \thalb A^* \quad\text{and}\quad
q_A \colon X\to\RR\colon x\mapsto \thalb \scal{x}{Ax}.
\end{equation}

\begin{fact}\emph{(See \cite[Proposition~2.3]{BWY2} and \cite[Proposition~2.2(v)]{BBW}).}\label{f1:Fitz}
Let $A\colon X\to X$ be linear and monotone, and
let $(x,x^*)\in X\times X$.
Then
\begin{equation}
F_A(x,x^*)=2
q_{A_+}^*(\tfrac{1}{2}x^*+\tfrac{1}{2}A^*x)=\tfrac{1}{2}q_{A_{+}}^*(x^*+A^*x).
\end{equation}
If $\ran A_{+}$ is closed, then $\dom q_{A_{+}}^*=\ran A_{+}$.
\end{fact}
To study Fitzpatrick functions of sums of maximal monotone operator, one needs the $\Box_{2}$ operation:
\begin{definition}
Let $F_1, F_2\colon X\times X\rightarrow\RX$.
Then the \emph{partial inf-convolution} $F_1\Box_2 F_2$
is the function defined on $X\times X$ by
\begin{equation*}F_1\Box_2 F_2\colon
(x,x^*)\mapsto \inf_{y^*\in X}\big(
F_1(x,x^*-y^*)+F_2(x,y^*)\big).
\end{equation*}
\end{definition}

\begin{fact}\label{psum}\emph{(See \cite[Lemma~23.9]{Si2} or \cite[Proposition~4.2]{BM}.)}
Let $A, B\colon X\To X$ be monotone such that $\dom A\cap\dom B\neq\varnothing$.
Then
$F_A\Box_2 F_{B}\geq F_{A+B}$.
\end{fact}
Under some constraint qualifications, one has
\begin{fact}\label{fitzpatrickfunctionsum}
\begin{enumerate}
\item \emph{(See {\cite{BBW}}.)}
Let $A, B: X\to X$ be continuous, linear, and
monotone operators such that $\ran(A_{+}+B_{+})$ is closed. Then
$F_{A+B}=F_{A}\Box_{2} F_{B}.$
\item \emph{(See {\cite{BWY3}}.)}
Let $A,B: X\To X$ be maximal
monotone linear relations, and suppose that $\dom A - \dom B$ is closed.
Then $F_{A+B}=F_{A}\Box_{2} F_{B}.$
\end{enumerate}
\end{fact}

\section{An unbounded skew operator on $\ell^{2}$}
\label{s:skew}
In this section, we construct a maximal monotone and skew operator $S$ on $\ell^{2}$ such that
$-S^*$ is not maximal monotone. This answers one of Svaiter's question.
We explicitly compute the Fitzpatrick functions $F_{S+S^*}$, $F_{S}$, $F_{S^*}$, and show
that $F_{S+S^*}\neq F_{S}\Box_{2} F_{S^*}$ even though $S,S^*$ are linear maximal monotone with $\dom S-\dom S^*$ being
a dense linear subspace in $\ell^{2}$.
\subsection{The Example in $\ell^2$}
Let $\ell^2$ denote the Hilbert space of real square-summable sequences
$(x_1,x_2,x_3,\ldots)$. 

\begin{example}\label{1EL:0}
Let $X=\ell^2$, and  $S:\dom S\rightarrow \ell^2$ be given by
\begin{align}Sy:=\frac{\bigg(\sum_{i< n}y_{i}-\sum_{i> n}y_{i}\bigg)}{2}
=\bigg(\sum_{i< n}y_{i}+\tfrac{1}{2}y_n\bigg),\quad \forall y=(y_n)\in\dom S,\label{EL:1}\end{align}
where $\dom S:=\big\{ y=(y_n)\in \ell^{2} \mid \sum_{i\geq 1}y_{i}=0,
 \bigg(\sum_{i\leq n}y_{i}\bigg)\in\ell^2\big\}$ and $\sum_{i<1}y_{i}=0$. 
In matrix form,
\begin{equation*}
S=\tfrac{1}{2}\begin{pmatrix}
0 & -1& -1&-1 &-1&\cdots&-1&-1&\cdots \\
1& 0& -1&-1&-1&\cdots &-1&-1&\cdots\\
1 &1 &0&  -1&-1&\cdots&-1&-1&\cdots \\
1 &1&1&0& -1&\cdots&-1&-1&\cdots\\
1 &1&1&1& 0&\cdots&-1&-1&\cdots\\
\vdots  & \ddots & \ddots  & \ddots& \ddots
\end{pmatrix},
\end{equation*}
or\begin{equation*}S=
\begin{pmatrix}
\tfrac{1}{2} & 0& 0&0 &0&\cdots&0&0&\cdots \\
1& \tfrac{1}{2}& 0&0 &0&\cdots &0&0&\cdots\\
1 &1 &\tfrac{1}{2}&  0&0&\cdots&0&0&\cdots \\
1 &1&1&\tfrac{1}{2}& 0&\cdots&0&0&\cdots\\
1 &1&1&1& \tfrac{1}{2}&\cdots&0&0&\cdots\\
\vdots  & \ddots & \ddots  & \ddots& \ddots
\end{pmatrix}.
\end{equation*}
{Using the second matrix,} 
it is easy to see that $S$ is injective.
\end{example}\label{EL:0}

\begin{proposition}\label{EL:p1}
Let $S$ be defined as in Example~\ref{1EL:0}. Then $S$ is skew.

\end{proposition}

\begin{proof}Let $y=(y_n)\in\dom S$. Then 
$\big(\sum_{i\leq n}y_{i}\big)\in\ell^2$.
 Thus,
\begin{align*}\ell^2\ni
\bigg(\sum_{i\leq n}y_{i}\bigg)-\tfrac{1}{2}y=\bigg(\sum_{i\leq n}y_{i}\bigg)
-\tfrac{1}{2}(y_n)=\bigg(\sum_{i< n}y_{i}+\tfrac{1}{2}y_n\bigg)=
Sy.\end{align*}
Hence $S$ is well defined.
Clearly, $S$ is linear on $\dom S$. Now we show $S$ is skew.

Let $y=(y_n)\in\dom S$, and $s:=\sum_{i\geq 1}y_{i}.$
Then $\bigg(\sum_{i\leq n}y_{i}\bigg)\in\ell^2$.
 Hence $ \bigg(\sum_{i< n}y_{i}\bigg)=\bigg(\sum_{i\leq n}y_{i}\bigg)-(y_n)\in\ell^2$.
By $s=0$, \begin{align}
&\ell^2\ni-\bigg(\sum_{i< n}y_{i}\bigg)=0-\bigg(\sum_{i< n}y_{i}\bigg)
=\bigg(\sum_{i\geq 1}y_{i}-\sum_{i< n}y_{i}\bigg)=
\bigg(\sum_{i\geq n}y_{i}\bigg),\nonumber\\
&\bigg(\sum_{i\geq n+1}y_{i}\bigg)=0-\bigg(\sum_{i\leq n}y_{i}\bigg)
\in\ell^2.\label{ELL:1}\end{align} Thus, by
\eqref{ELL:1},
\begin{align}
-2\scal{Sy}{y}&=\langle \bigg(\sum_{i> n}y_{i}-\sum_{i< n}y_{i}\bigg),y\rangle
=\langle \bigg(\sum_{i\geq n+1}y_{i}+\sum_{i\geq n}y_{i}\bigg),y\rangle\label{SE:1}\\
 & =\scal{\bigg(\sum_{i\geq 1}y_{i},\sum_{i\geq 2}y_{i},\cdots\bigg)
 +\bigg(\sum_{i\geq 2}y_{i},\sum_{i\geq 3}y_{i},\cdots\bigg)}{y}\nonumber\\
&= \scal{(s,s-y_{1},s-(y_{1}+y_{2}),\cdots)+(s-y_{1},
s-(y_{1}+y_{2}),\cdots)}{(y_{1},y_{2},\cdots)}\nonumber\\
& =[sy_{1}+(s-y_{1})y_{2}+(s-(y_{1}+y_{2}))y_{3}+\cdots]+\nonumber\\
&\quad \quad[(s-y_{1})y_{1}+(s-(y_{1}+y_{2}))y_{2}+(s-(y_{1}+y_{2}+y_{3}))y_{3}+\cdots]\nonumber\\
&=\lim_{n}[sy_{1}+(s-y_{1})y_{2}+\cdots +(s-(y_{1}+\cdots+y_{n-1}))y_{n}]+ \nonumber\\
&\quad\quad \lim_{n}
[(s-y_{1})y_{1}+(s-(y_{1}+y_{2}))y_{2}+\cdots +(s-(y_{1}+\cdots+y_{n}))y_{n}]\nonumber\\
&=\lim_{n}[s(y_{1}+\cdots+y_{n})-y_{1}y_{2}-(y_{1}+y_{2})y_{3}-\cdots -(y_{1}+\cdots+y_{n-1})y_{n}]
   +\nonumber\\
   &\quad\quad [s(y_{1}+\cdots +y_{n})-(y_{1}^{2}+\cdots +y_{n}^2)-y_{1}y_{2}
   -\cdots -(y_{1}+\cdots +y_{n-1})y_{n}]\nonumber\\
& =\lim_{n}[2s(y_{1}+\cdots +y_{n})-(y_{1}+\cdots +y_{n})^2]
=2s^2-s^2=s^2=0.\nonumber
\end{align}
Hence $S$ is skew.
\end{proof}

\begin{remark}
$S$ is unbounded in Example~\ref{1EL:0}, since $e:=(1,0,0,\cdots,0,\cdots)\notin\dom S$.
\end{remark}

\begin{fact}[Phelps-Simons]\emph{(See. \cite[Proposition~3.2(a)]{PheSim})}.
\label{PF:1} Let $S:\dom S\rightarrow X$ be linear and monotone.
 Then $(x,x^*)\in X\times
X$ is monotonically related to $\gra S$ if, and only if
\begin{equation*}
\langle x,x^*\rangle\geq 0\; \text{and}\;
 \left[\langle Sy,x\rangle+\langle x^*,y\rangle\right]^2
\leq 4\langle x^*,x\rangle\langle Sy,y\rangle,\quad \forall y\in \dom S.\end{equation*}
\end{fact}

\begin{proposition}\label{EL:0L}
Let $S$ be defined in Example~\ref{1EL:0}.
Then $S$ is a maximal monotone operator. In particular, $\gra S$ is closed.
\end{proposition}

\begin{proof}
By Proposition~\ref{EL:p1}, $S$ is skew. Let $(x,x^*)\in X\times X$ 
be monotonically
related to $\gra S$. Write $x=(x_n)$ and $x^*=(x^*_n)$.
By Fact~\ref{PF:1}, we have
\begin{align}
&\langle Sy,x\rangle+\langle x^*,y\rangle=0, \quad \forall y\in\dom S.\label{EL:3}\end{align}
Let $e_n=(0,\ldots,0,1,0,\ldots):$ the $n$th entry is $1$ and the
others are $0.$ Then let $y=-e_1+e_n$. Thus $y\in\dom S$
 and $Sy=(-\tfrac{1}{2},-1,\ldots,-1,-\tfrac{1}{2},0,\ldots)$.
Then by \eqref{EL:3},
\begin{align}
&-x^*_1+x^*_n-\tfrac{1}{2}x_1-\tfrac{1}{2}x_n-\sum_{i=2}^{n-1}x_i=0
\Rightarrow x^*_n=x^*_1-\tfrac{1}{2}x_1+\sum_{i=1}^{n-1}x_i+\tfrac{1}{2}x_n.
\label{EL:4}
\end{align}
Since $x^*\in\ell^2$ and $x\in\ell^2$, 
we have $x^*_n\rightarrow0, x_n\rightarrow0$.
Thus by \eqref{EL:4}, \begin{align}-\sum_{i\geq1}x_i=x^*_1-\tfrac{1}{2}x_1.\label{EL:5}\end{align}
Next we show $-\sum_{i\geq1}x_i=x^*_1-\tfrac{1}{2}x_1=0$.
Let $s=\sum_{i\geq1}x_i$. Then  by \eqref{EL:4} and \eqref{EL:5},
\begin{align}
&2x^*=2(x^*_n)=2\bigg(-\sum_{i\geq1}x_i+\sum_{i<n}x_i+\tfrac{1}{2}x_n\bigg)=
\bigg(-2\sum_{i\geq1}x_i+2\sum_{i<n}x_i+x_n\bigg)\nonumber\\
&=\bigg(-2\sum_{i\geq n}x_i+x_n\bigg)=\bigg(-\sum_{i\geq n}x_i-\sum_{i\geq n}x_i+x_n\bigg)\nonumber\\
&=\bigg(-\sum_{i\geq n}x_i-\sum_{i\geq n+1}x_i\bigg).\label{EL:6}\end{align}
On the other hand, by \eqref{EL:4},
\begin{align*}\ell^2\ni x^*-\tfrac{1}{2}x&=\bigg(-\sum_{i\geq1}x_i
+\sum_{i<n}x_i+\tfrac{1}{2}x_n\bigg)-(\tfrac{1}{2}x_n)
=\bigg(-\sum_{i\geq n }x_i\bigg).\end{align*}
Then by \eqref{EL:6},
\begin{align*}
&2x^*=\bigg(-\sum_{i\geq n}x_i\bigg)+\bigg(-\sum_{i\geq n+1}x_i\bigg).\end{align*}
Then by Fact~\ref{PF:1},  similar to the proof in \eqref{SE:1} in Proposition~\ref{EL:0},  we have
\begin{align*}
0\geq-2\scal{x^*}{x}&
=\langle \bigg(\sum_{i\geq n}x_i\bigg)+\bigg(\sum_{i\geq n+1}x_i\bigg),x\rangle\\
 & =\scal{\bigg(\sum_{i\geq 1}x_{i},\sum_{i\geq 2}x_{i},\cdots\bigg)
 +\bigg(\sum_{i\geq 2}x_{i},\sum_{i\geq 3}x_{i},\cdots\bigg)}{x}\\
&=2s^2-s^2=s^2.
\end{align*}
Hence $s=0$, i.e., $x^*_1=\tfrac{1}{2}x_1$.
 By \eqref{EL:4},
$x^*=\bigg(\sum_{i<n}x_i+\tfrac{1}{2}x_n\bigg)$. Thus
\begin{align*}
\ell^2\ni x^*+\tfrac{1}{2}x=\bigg(\sum_{i<n}x_i+\tfrac{1}{2}x_n\bigg)+\big(\tfrac{1}{2}x_n\big)=
\bigg(\sum_{i\leq n}x_i\bigg).\end{align*}
Hence $x\in\dom S$ and $x^*=Sx$. Thus, $S$ is maximal monotone. Hence $\gra S$ is closed.
\end{proof}

\begin{proposition}\label{PE:01a}
Let  $S$ be defined in Example~\ref{1EL:0}. Then
\begin{align}S^*y
=\bigg(\sum_{i> n}y_{i}+\tfrac{1}{2}y_n\bigg),\quad \forall y=(y_n)\in\dom S^*,\label{PF:a2}\end{align}
where $\dom S^*=\big\{ y=(y_n)\in \ell^{2} \mid \sum_{i\geq 1}y_{i}\in \RR,
 \bigg(\sum_{i> n}y_{i}\bigg)\in\ell^2\big\}.$ In matrix form,
\begin{equation*}
S^*:=\begin{pmatrix}
\tfrac{1}{2} & 1& 1&1 &1&\cdots&1&1&\cdots \\
0& \tfrac{1}{2}& 1&1&1&\cdots &1&1&\cdots\\
0&0 &\tfrac{1}{2}&  1&1&\cdots&1&1&\cdots \\
0&0&0&\tfrac{1}{2}& 1&\cdots&1&1&\cdots\\
0 &0&0&0& \tfrac{1}{2}&\cdots&1&1&\cdots\\
\vdots  & \ddots & \ddots  & \ddots& \ddots & \ddots &\cdots &\cdots
\end{pmatrix}.
\end{equation*}
Moreover, 
$\dom S\subsetneqq\dom S^*$, 
$S^*=-S$ on $\dom S$, and $S^*$ is not skew.
\end{proposition}
\begin{proof}
Let $y=(y_n)\in \ell^{2}$ with $\bigg(\sum_{i> n}y_{i}\bigg)\in\ell^2$,
 and $y^*=\bigg(\sum_{i> n}y_{i}+\tfrac{1}{2}y_n\bigg)$.
Now we show $(y,y^*)\in\gra S^*$.
Let $s=\sum_{i\geq1} y_{i}$ and $x\in\dom S$. Then we have
\begin{align*}
&\langle y, Sx\rangle+\langle y^*, -x\rangle
=\langle y, \tfrac{1}{2}x+\bigg(\sum_{i< n}x_{i}\bigg)\rangle
+\langle \tfrac{1}{2}y+\bigg(\sum_{i> n}y_{i}\bigg), -x\rangle\\
&=\langle y, \bigg(\sum_{i< n}x_{i}\bigg)\rangle
+\langle \bigg(\sum_{i> n}y_{i}\bigg), -x\rangle\\
&=\lim_{n} \left[y_2 x_1+y_3 (x_1+x_2)+\cdots+y_n (x_1+\cdots+x_{n-1})\right]\\&\quad
-\lim_{n} \left[ x_1 (s-y_1)+x_2 (s-y_1-y_2)+\cdots+x_n (s-y_1-\cdots-y_{n})\right]\\
&=\lim_{n} \left[ x_1 (y_2+\cdots+y_n)+x_2 (y_3+\cdots+y_n)+\cdots+x_{n-1} y_n\right]\\&\quad
-\lim_{n} \left[ x_1 (s-y_1)+x_2 (s-y_1-y_2)+\cdots+x_n (s-y_1-\cdots-y_{n})\right]\\
&=\lim_{n} \left[ x_1 (y_1+y_2+\cdots+y_n-s)+x_2  (y_1+y_2+\cdots+y_n-s)+\cdots+x_{n}  (y_1+y_2+\cdots+y_n-s)\right]\\
&=\lim_{n} \left[ (x_1+\cdots+x_n) (y_1+y_2+\cdots+y_n-s)\right]\\
&=0.
\end{align*}
Hence $(y,y^*)\in\gra S^*$.

On the other hand,
let $(a,a^*)\in\gra S^*$ with $a=(a_n)$ and $a^*=(a^*_n)$.
Now we show
\begin{align}
\bigg(\sum_{i> n}a_{i}\bigg)\in\ell^2\; \text{and}\;
a^*=\bigg(\sum_{i> n}a_{i}+\tfrac{1}{2}a_n\bigg).\label{PF:a1}
\end{align}
Let $e_n=(0,\cdots,0,1,0,\cdots):$ the $n$th entry is $1$ and the
others are $0.$ Then let $y=-e_1+e_n$. Thus $y\in\dom S$
 and $Sy=(-\tfrac{1}{2},-1,\cdots,-1,-\tfrac{1}{2},0,\cdots)$.
Then,
\begin{align}
0&=\langle a^*,y\rangle+\langle -Sy,a\rangle
=-a^*_1+a^*_n+\tfrac{1}{2}a_1+\tfrac{1}{2}a_n+\sum_{i=2}^{n-1}a_i\nonumber\\
&\Rightarrow a^*_n=a^*_1-\tfrac{1}{2}a_1-\sum_{i=2}^{n-1}a_i-\tfrac{1}{2}a_n.\label{PEL:a1}
\end{align}

Since $a^*\in\ell^2$ and $a\in\ell^2$, $a^*_n\rightarrow0, a_n\rightarrow0$.
Thus by \eqref{PEL:a1}, \begin{align}a^*_1=\tfrac{1}{2}a_1+\sum_{{i>1}}a_i,\label{PEL:a2}\end{align}
from which we see that $\sum_{i\geq1}a_{i}\in \RR$.
Combining \eqref{PEL:a1} and \eqref{PEL:a2}, we have
\begin{align*}
a^*_n=\sum_{i> n}a_{i}+\tfrac{1}{2}a_n
\end{align*}
Thus,
\eqref{PF:a1} holds.
Hence \eqref{PF:a2} holds.

Now for $x\in \dom S$, since $\sum_{i\geq1}x_{i}=0$, we have 
\begin{align*}
S^*x &=\bigg(\tfrac{1}{2}x_{n}+\sum_{i>n}x_{i}\bigg)=\bigg(-\tfrac{1}{2}x_{n}+\sum_{i\geq n}x_{i}\bigg)\\
&=\bigg(-\tfrac{1}{2}x_{n}-\sum_{i<n}x_{i}\bigg)=-Sx.
\end{align*}
We note that $S^*$ is not skew since for $e_{1}=(1,0,\cdots),$ $\scal{S^*e_{1}}{e_{1}}=\scal{1/2 e_{1}}{e_{1}}=1/2.$
As $e=(1,0,0,\cdots,0,\cdots) \in\dom S^*$ but $e\not\in\dom S$.
we have $\dom S\subsetneqq\dom S^*$.
\end{proof}

\begin{proposition}\label{PE:02a}
Let $S$ be defined in Example~\ref{1EL:0}.
Then \begin{align}
\langle S^*y, y\rangle=\tfrac{1}{2}s^2,
\quad \forall y\in\dom S^*\ \text{with}\quad s=\sum_{i\geq1} y_i.\label{PE:sa1}\end{align}

\end{proposition}
\begin{proof}
Let $y=(y_n)\in\dom S^*$, and $s=\sum_{i\geq1} y_i$. 
By Proposition~\ref{PE:01a}, we have
$s\in\RR$ and 
\begin{align*}
&\langle S^*y,y\rangle=\langle \bigg(\sum_{i> n}y_{i}+\tfrac{1}{2}y_n\bigg),y\rangle
=\langle \bigg(\sum_{i\geq n}y_{i}-\tfrac{1}{2}y_n\bigg),y\rangle\\
&=\lim_{n} \left[ sy_1+ (s-y_1)y_2
+\cdots+ (s-y_1-y_2-\cdots-y_{n-1})y_n-\tfrac{1}{2}(y^2_1+y^2_2+\cdots+y^2_n)\right]\\
&=\lim_{n} \left[ s(y_1+\cdots+y_n)-  y_1y_2-(y_1+y_2)y_3
-\cdots-(y_1+y_2+\cdots+y_{n-1})y_n \right]\\
&\quad-\tfrac{1}{2}\left[y^2_1+y^2_2+\cdots+y^2_n\right]\\
&=\lim_{n} \left[ s(y_1+\cdots+y_n)\right]\\&\quad
-\lim_{n} \left[y_1y_2+(y_1+y_2)y_3
+\cdots+(y_1+y_2+\cdots+y_{n-1})y_n +\tfrac{1}{2}(y^2_1+y^2_2+\cdots+y^2_n)\right]\\
&=s^2
-\lim_{n}\tfrac{1}{2} \left[y_1+y_2+\cdots+y_{n}\right]^2\\
&=s^2-\tfrac{1}{2}s^2\\
&=\tfrac{1}{2}s^2.
\end{align*}
Hence \eqref{PE:sa1} holds.
\end{proof}

\begin{proposition}\label{EL:10}
Let $S$ be defined in Example~\ref{1EL:0}.
Then $-S$ is not maximal monotone.
\end{proposition}
\begin{proof}
By Proposition~\ref{EL:p1}, $-S$ is skew.
Let $e=(1,0,0,\cdots,0,\cdots)$. Then $e\notin\dom S=\dom (-S)$. Thus,
 $(e,\tfrac{1}{2}e)\notin\gra (-S)$.
 We have for every $y\in\dom S$, 
\begin{align*}
\langle e,\tfrac{1}{2}e\rangle\geq0\;\text{and}\;
\langle e, -Sy\rangle+\langle y,\tfrac{1}{2}e\rangle=-\tfrac{1}{2}
y_1+\tfrac{1}{2}
y_1=0.\end{align*}
By Fact~\ref{PF:1},
$(e,\tfrac{1}{2}e)$ is monotonically related
to $\gra (-S)$. Hence $-S$ is not maximal monotone.
\end{proof}

We proceed to show that for every 
maximal monotone and skew operator $S$, the operator $-S$ has
a unique maximal monotone extension, namely $S^*$.

\begin{theorem}\label{Lm:1}
Let $S: \dom S\rightarrow X$ be a maximal monotone skew operator. Then
$-S$ has a unique maximal monotone extension: $S^*$.
\end{theorem}
\begin{proof}
By Fact~\ref{EL:11}, $\gra (-S)\subseteq\gra S^*$.
Assume $T$ is a maximal monotone extension of $-S$. Let $(x,x^*)\in\gra T$.
Then $(x,x^*)$ is
monotonically related to $\gra(-S)$. By
Fact~\ref{PF:1},
\begin{align*}
\langle x^*,y\rangle+\langle -x,Sy\rangle
=\langle x^*,y\rangle+\langle x,-Sy\rangle=0,\quad \forall y\in\dom S.\end{align*}
Thus $(x,x^*)\in\gra S^*$. Since $(x,x^*)\in\gra T$ is arbitrary, we have $\gra T\subseteq\gra S^*$.
By Fact~\ref{Sv:7}, $S^*$ is maximal monotone. Hence $T=S^*$.
\end{proof}

\begin{remark} Note that \cite[Proposition~17]{VZ08} also implies
that  $-S$ has a unique maximal monotone extension, where $S$ is as in
Theorem~\ref{Lm:1}.
\end{remark}

\begin{remark}
Define the \emph{right and left shift operators} $R, L:\ell^2\rightarrow \ell^2$ by
$$Rx=(0,x_{1},x_{2},\ldots),\quad  Lx=(x_{2},x_{3},\ldots),
\quad \forall \ x=(x_{1},x_{2},\ldots)\in \ell^2.$$ One can verify that in Example~\ref{1EL:0}
$$S=(\Id-R)^{-1}-\frac{\Id}{2}, \quad S^*=(\Id-L)^{-1}-\frac{\Id}{2}.$$
The maximal monotone operators $(\Id-R)^{-1}$ and $(\Id-L)^{-1}$ have been utilized by
Phelps and Simons, see \cite[Example 7.4]{PheSim}.
\sethlcolor{yellow}
{Should we include more details? Can you show me the details at least?
What about pointing out that $R^*=L$ etc?}
\end{remark}

\subsection{An answer to Svaiter's question}

\begin{definition}\label{Mc:1}
Let $S: X\rightrightarrows X$ be skew. We say $S$ is \emph{maximal skew}
 $($termed \emph{``maximal self-cancelling"
in \cite{SV}}$)$
if no proper enlargement
(in the sense of graph inclusion) of $S$ is skew. We say $T$ is \emph{a maximal skew extension}
of $S$ if $T$ is maximal skew and $\gra T\supseteq\gra S$.
\end{definition}

\begin{lemma}\label{Sl:1}
Let $S: X\rightrightarrows X$ be a maximal monotone skew operator.
 Then
both $S$ and $-S$ are maximal skew.
\end{lemma}

\begin{proof}
Clearly, $S$ is maximal skew. Now we show $-S$ is maximal skew.
Let $T$ be a skew operator such that $\gra(-S)\subseteq\gra T$.
Thus, $\gra S\subseteq\gra (-T)$. Since $-T$ is monotone
 and $S$ is maximal monotone,
$\gra S=\gra(-T)$. Then $-S=T$. Hence $-S$ is maximal skew.
\end{proof}

\begin{fact}[Svaiter]\label{Sv:9}\emph{(See \cite{SV}.)}
Let $S: X\rightrightarrows X$ be maximal skew.
 Then either $-S^*(\text{i.e.,}\;S^{\vdash})$
 or $S^*(\text{i.e.,}\;-S^{\vdash})$ is maximal monotone.
\end{fact}

In \cite{SV},
Svaiter asked
whether or not $-S^*(\text{i.e.,}\;S^{\vdash})$ is maximal monotone if $S$ is maximal skew.
Now we can give a negative answer, even though $S$ is maximal monotone and skew.

\begin{theorem}\label{Pl:1}
Let  $S$ be defined in Example~\ref{1EL:0}. Then $S$ is maximal skew,
but $-S^*$ is not monotone, so not maximal monotone.
\end{theorem}
\begin{proof} Let $e=(1,0,0,\cdots,0,\cdots)$. By Proposition~\ref{PE:01a},
 $(e,-\tfrac{1}{2}e)\in\gra (-S^*)$, but $\langle e,-\tfrac{1}{2}e\rangle
 =-\tfrac{1}{2}<0$. Hence $-S^*$ is not monotone.
\end{proof}

 By Theorem~\ref{Pl:1},
$-S^*(\text{i.e.,}\;S^{\vdash})$ is not always maximal monotone.
 Can one improve
 Svaiter's result: ``If $S$ is maximal skew,
  then $S^*$ (\text{i.e.,}\;$-S^{\vdash}$)
is always maximal monotone?"

\begin{theorem}There exists a maximal skew operator $T$ on $\ell^2$
such that $T^*$ is not maximal monotone. Consequently, 
Svaiter's result is optimal.
\end{theorem}

\begin{proof}
Let $T=-S$,
where $S$ be defined in Example~\ref{1EL:0}.
By Lemma~\ref{Sl:1}, $T$ is maximal skew.
Then by Theorem~\ref{Pl:1} and Fact~\ref{Rea:1}\ref{MPP:4},
 $T^*=(-S)^*=-S^*$ is not maximal monotone.
Hence  Svaiter's result cannot be further improved.
\end{proof}

\subsection{The maximal monotonicity and Fitzpatrick functions of a sum}

\begin{example}[$S+S^*$ fails to be maximal monotone] \label{RE:1}
Let $S$ be defined  in Example~\ref{1EL:0}.
Then neither $S$ nor $S^*$ has full domain.
By Fact~\ref{EL:11},
$\forall x\in \dom (S+S^*)=\dom S$, we have
$$(S+S^*)x=0.$$
Thus $S+S^*$ has a proper monotone extension
 from $\dom (S+S^*)$ to the $0$ map on $X$.
Consequently, $S+S^*$ is not maximal monotone.
 This supplies a different example for showing that the constraint
qualification in the sum problem of maximal
 monotone operators can not be substantially weakened, see \cite[Example
7.4]{PheSim}.
\end{example}

We now compute $F_{S}, F_{S^*}, F_{S+S^*}$. As a result, we see that
$F_{S+S^*}\neq F_{S}\Box_{2}F_{S^*}$ even though $S,S^*$ are maximal monotone
with $\dom S-\dom S^*$ being dense in $\ell^{2}$.
Since $\ran(S_{+}+(S^*)_{+})=\{0\}$ and $F_{S+S^*}\neq F_{S}\Box_{2}F_{S^*}$, this also
means that
Fact~\ref{fitzpatrickfunctionsum}(i) fails for discontinuous linear maximal monotone operators.

\begin{lemma}\label{fitzfors}
{Let $S:\dom S\to X$ be a maximal monotone skew linear operator.} 
Then
$$F_{S}=\iota_{\gra (-S^*)},$$
$$F_{S^*}^{*\intercal}=F_{S^*}=\iota_{\gra S^*}+\scal{\cdot}{\cdot}.$$
\end{lemma}
\begin{proof}
By \cite[Proposition~5.5]{BWY3}, \begin{align*}F^*_S=
(\iota_{\gra S})^{\intercal}.\label{2S:1}\end{align*}
Then \begin{align}F_{S}=\big(F_{S}^{*\intercal}\big)^{*\intercal}=
\big(\iota_{\gra S}\big)^{*\intercal}=
\big(\iota_{\gra S}^{\intercal}\big)^*=
\big(\iota_{\gra S^{-1}}\big)^*=\iota_{(\gra S^{-1})^{\perp}}
=\iota_{\gra (-S^*)}.\end{align}
From Fact~\ref{EL:11}, $\gra -S\subseteq \gra S^*$, we have
$$F_{S^*}\geq F_{-S}=\iota_{\gra -(-S)^*}=\iota_{\gra S^*},$$
this shows that $\dom F_{S^*}\subseteq \gra S^*$. By Fact~\ref{f:Fitz},
$F_{S^*}(x,x^*)=\scal{x}{x^*}\ \forall
(x,x^*)\in \gra S^*$. Hence
$F_{S^*}=\iota_{\gra S^*}+\scal{\cdot}{\cdot}$.
Again by \cite[Proposition~5.5]{BWY3},
$F^{*\intercal}_{S^*}=\iota_{\gra S^*}+\scal{\cdot}{\cdot}$. 
\end{proof}

\begin{theorem}\label{2FI:7}
Let $S$ be defined as in Example~\ref{1EL:0}. Then
\begin{align}
F_{S+S^*}(x,x^*)&=\iota_{X\times\{0\}}(x,x^*)\nonumber\\
F_S\Box_2F_{S^*}(x,x^*)&=\begin{cases}
\tfrac{1}{2}s^2,\;&\text{if}\; (x,x^*)\in\dom S^*\times \{0\}\,\text{with}\,s=\sum_{i\geq1}x_i;\\
\infty\;&\text{otherwise}.\end{cases}\label{PSum:1}
\end{align}
Consequently,
$F_S\Box_2F_{S^*}\neq F_{S+S^*}.$
\end{theorem}

\begin{proof}
By Fact~\ref{EL:11}, \begin{align}(S+S^*)|_{\dom S}=0.\label{sumP:4}\end{align}

Let $(x,x^*)\in X\times X$. Using \eqref{sumP:4} and
Fact~\ref{EL:11}, we have
\begin{equation}
F_{S+S^*}(x,x^*)=\sup_{a\in\dom S}\langle x^*,a\rangle
=\iota_{(\dom S)^{\bot}}(x^*)
=\iota_{\{0\}}(x^*)=\iota_{X\times\{0\}}(x,x^*).\label{22FI:5}
\end{equation}
Then by Fact~\ref{psum}, we have
\begin{align}
 F_S\Box_2F_{S^*}(x,x^*)=\infty,\quad x^*\neq 0.\label{1sump:1}\end{align}

It follows from Lemma~\ref{fitzfors} that
\begin{align}
 F_S\Box_2F_{S^*}(x,0)&=\inf_{y^*\in X}\{
F_S(x,y^*)+F_{S^*}(x,-y^*)\}\nonumber\\
&=\inf_{y^*\in X}\{\iota_{\gra (-S^*)}(x,y^*)
+\iota_{\gra S^*}(x,-y^*)+\langle x,-y^*\rangle\}\nonumber\\
&=\inf_{y^*\in X}\{\iota_{\gra S^*}(x,-y^*)+\langle x,-y^*\rangle\}.\label{2FI:3}\end{align}
Thus, $F_S\Box_2F_{S^*}(x,0)=\infty$ if $x\notin\dom S^*$.
Now suppose $x\in\dom S^*$ and $s=\sum_{i\geq1} x_i$.
Then by \eqref{2FI:3} and Proposition~\ref{PE:02a},
 we have
\begin{align*}F_S\Box_2F_{S^*}(x,0)=\langle x, S^*x\rangle=\tfrac{1}{2}s^2.
\end{align*}
Combine the results above,  \eqref{PSum:1} holds.
Since $\dom S^*\neq X$, $F_S\Box_2F_{S^*}\neq F_{S+S^*}.$
\end{proof}

\begin{remark}
\cite[Theorem~7.6]{BWY3} shows that: 
Let $A:X\To X$ be a maximal monotone linear relation.
Then $A^*=-A$ if and only if $\dom A = \dom A^*$ and $F_{A}=F_{A}^{*\intercal}.$
Let $A=S^*$ with $S$ defined as in Example~\ref{1EL:0}. Lemma~\ref{fitzfors} shows that
$F_{A}=F_{A}^{*\intercal}$, but $A^*=S\neq -S^*=-A$. Hence the requirement $\dom A=\dom A^*$ can not be
omitted.
\end{remark}

\section{The inverse Volterra operator on $L^2[0,1]$}\label{s:vol}
Let $V$ be the Volterra integral operator. 
In this section, we systematically study $T = V^{-1}$ and
its skew part $S:=\thalb(T-T^*)$. 
 It turns out that $T$ is neither skew nor symmetric and that its skew part $S$
  admits two maximal monotone and skew extensions $T_{1}, T_{2}$ (in fact, anti-self-adjoint)
   even though $\dom S$ is a dense linear subspace of $L^{2}[0,1]$.
This will give another
simpler example of Phelps-Simons' showing that the constraint
qualification for the  sum of monotone
 operators cannot be significantly weakened, 
 see \cite[Theorem~5.5]{SiZ} or \cite{Voisei06}.
 We compute the Fitzpatrick functions
  $F_{T}, F_{T^*}$, $F_{T+T^*}$,  
and we show that $F_{T}\Box_{2} F_{T^*}\neq F_{T+T^*}$.
This shows that the constraint qualification
for the formula of the Fitzpatrick function of the sum of
two maximal monotone operators cannot be significantly weakened either.

\begin{definition}[\cite{BWY3}] Let $T:X\rightrightarrows X$ be a linear relation.
 We say that $T$ is symmetric  if
$\gra T\subseteq\gra T^*$; $T$ is
self-adjoint if $T^*=T$ and anti-self-adjoint if $T^*=-T$.
\end{definition}

\subsection{Properties of the Volterra operator and its inverse}

To study the Volterra operator and its inverse, we shall frequently need
the following generalized integration-by-parts formula, see
\cite[Theorem 6.90]{stromberg}.
\begin{fact}[Generalized integration by parts]\label{bypart}
Assume that $x,y$ are absolutely continuous functions on the 
interval $[a,b]$. Then
$$\int_{a}^b xy'+\int_{a}^{b}x'y=x(b)y(b)-x(a)y(a).
$$
\end{fact}
Fact~\ref{Sv:7} allows us to claim that
\begin{proposition}\label{skewmax}
Let $A:X\To X$ be a linear relation. If $A^*=-A$, then both $A$
and $-A$ are maximal monotone and skew.
\end{proposition}
\begin{proof}
Since $A=-A^*$, we have that $\dom A=\dom A^*$ and that $A$ has closed
graph. Now $\forall x\in \dom A$, by Fact~\ref{Rea:1}\ref{Sia:2b},
$$\scal{Ax}{x}=\scal{x}{A^*x}=-\scal{x}{Ax} \quad\Rightarrow\quad \scal{Ax}{x}=0.$$
Hence $A$ and $-A$ are skew. As $A^*=-A$ is monotone,
Fact~\ref{Sv:7} shows that $A$ is maximal monotone.

Now $-A=A^*=-(-A)^*$ and $-A$ is a linear relation. Similar
arguments show that $-A$ is maximal monotone.
\end{proof}

\begin{example}[Volterra operator]\label{ex:Volterra}(See \cite[Example~3.3]{BBW}.)
Set $X=L^2[0,1]$. The \emph{Volterra integration
operator} \cite[Problem~148]{Halmos} is defined by
\begin{equation} \label{e:aug17:a}
V \colon X \to X \colon x \mapsto Vx, \quad\text{where}\quad
Vx\colon[0,1]\to\RR\colon t\mapsto \int_{0}^{t}x,
\end{equation}
and its adjoint is given by
$$
t\mapsto (V^*x)(t) = \int_{t}^{1}x,\quad\forall x \in X.
$$
Then
\begin{enumerate}
\item\label{readingbreak}Both $V$ and $V^*$ are maximal monotone
since they are monotone, continuous and linear.

\item\label{V:001} Both ranges
\begin{equation}\label{rangeone}
\ran V=\{x\in L^{2}[0,1]:\ \mbox{ $x$ is absolutely continuous},
x(0)=0, x'\in L^{2}[0,1]\},
\end{equation} and
\begin{equation}\label{rangetwo}
\ran V^*=\{x\in L^{2}[0,1]:\ \mbox{ $x$ is absolutely continuous},
x(1)=0, x'\in L^{2}[0,1]\},
\end{equation}
are dense in $X$,
 and both $V$ and $V^*$ are one-to-one.

\item\label{V:0} $\ran V\cap\ran V^*=\{Vx~\mid x\in e^{\bot}\}$, where $e\equiv 1\in L^2[0,1]$.

\item \label{V:010} 
Define $V_{+}x :=\thalb(V+V^*)(x)=\frac{1}{2}\scal{e}{x}e$.
Then $V_{+}$ is self-adjoint and
$$\ran V_{+}=\spand\{e\}.$$

\item \label{FI:4} Define $V_{\circ}x:=\thalb(V-V^*)(x):t\mapsto
\tfrac{1}{2}[\int_{0}^t x-\int_{t}^{1}x]\quad \forall x\in
L^{2}[0,1], t\in [0,1]$.
Then $V_\circ$ is anti-self-adjoint and 
\begin{align*}\ran V_{\circ}=\{x\in L^{2}[0,1]: \mbox{ $x$ is
absolutely continuous on $[0,1]$, $x'\in L^{2}[0,1],
x(0)=-x(1)$}\}.\end{align*}
\end{enumerate}
\end{example}
\begin{proof}
(i) By Fact~\ref{bypart},
$$\scal{x}{Vx}=\int_{0}^{1}x(t)\int_{0}^{t}x(s)ds dt=\frac{1}{2}
\bigg(\int_{0}^{1}x(s)ds\bigg)^2\geq 0,$$
so $V$ is monotone.

As $\dom V= L^{2}[0,1]$ and $V$ is continuous, $\dom
V^*=L^{2}[0,1]$. Let $x, y\in L^{2}[0,1]$. We have
\begin{align*}
\scal{Vx}{y}& =\int_{0}^{1}\int_{0}^{t}x(s)ds y(t) dt
=\int_{0}^{1}x(t)dt\int_{0}^{1}y(s)ds-\int_{0}^{1}\int_{0}^{t}y(s)ds x(t) dt\\
& =
\int_{0}^{1}\ \bigg(\int_{0}^{1}y(s)ds-\int_{0}^{t}y(s)ds\bigg)x(t) dt=
\int_{0}^{1}\int_{t}^{1}y(s)ds x(t) dt
=\scal{V^*y}{x},
\end{align*}
thus $(V^*y)(t)=\int_{t}^{1}y(s)ds$ $\forall t\in [0,1]$.

(ii) To show \eqref{rangeone}, if $z\in \ran V$, then
$$z(t)=\int_{0}^{t}x \quad \mbox{ for some $x\in L^{2}[0,1]$},$$
and hence $z(0)=0$, $z$ is absolutely continuous, and $z'=x \in L^{2}[0,1]$.
On the other hand, if $z(0)=0$, $z$ is absolutely continuous, $z'\in L^{2}[0,1]$, then
$z=Vz'$.

To show \eqref{rangetwo}, if $z\in \ran V^*$, then
$$z(t)=\int_{t}^{1}x\quad \mbox{ for some $x\in L^{2}[0,1]$},$$ and
hence $z(1)=0$, $z$ is a absolutely continuous, and $z'=-x\in
L^{2}[0,1]$. On the other hand, if $z(1)=0$, $z$ is absolutely
continuous, $z'\in L^{2}[0,1]$, then $z=V^*(-z').$

\ref{V:0} follows from \ref{V:001} (or see \cite{BBW}).

\ref{V:010} is clear.

\ref{FI:4}  If $x$ is absolutely continuous,
$x(0)=-x(1)$, $x'\in L^{2}[0,1]$, we have

\begin{align*}V_{\circ}x'(t)=\tfrac{1}{2}\bigg(\int_{0}^{t}x'-\int_{t}^{1}x'\bigg)
=\tfrac{1}{2}\bigg(x(t)-x(0)-x(1)+x(t)\bigg)=x(t).\end{align*}
This shows that $x\in \ran V_{\circ}$. Conversely, if $x\in \ran
V_{\circ}$, i.e.,
$$x(t)=\frac{1}{2}\int_{0}^{t}y-\frac{1}{2}\int_{t}^{1}y\quad \mbox{ for some $y\in L^{2}[0,1]$},$$
then $x$ is absolutely continuous, $x'=y\in L^{2}[0,1]$ and
$x(0)=-x(1)=-\frac{1}{2}\int_{0}^{1}y.$
\end{proof}

\begin{theorem}[Inverse Volterra operator=Differentiation operator]\label{Eth:1}
Let $X=L^2[0,1]$, and $V$ be the Volterra integration operator.
We let $T=V^{-1}$ and $D = \dom T \cap \dom T^*$. 
Then the following hold. 
\begin{enumerate}
\item\label{V:1} $T:\dom T\rightarrow X$ is given by
$Tx=x'$ with
$$\dom T=\{x\in L^{2}[0,1]:\ \mbox{ $x$ is absolutely continuous},
x(0)=0, x'\in L^{2}[0,1]\},$$
and $T^*:\dom T^*\rightarrow X$ is given by
$T^*x=-x'$ with
$$\dom T^*=\{x\in L^{2}[0,1]:\ \mbox{ $x$ is absolutely continuous},
x(1)=0, x'\in L^{2}[0,1]\}.$$ Both $T$ and $T^*$ are maximal
monotone linear operators. 

\item\label{V:3} $T$ is neither skew  nor symmetric.

\item\label{V:4} 
The linear subspace
$$D=\big\{x\in L^{2}[0,1]:\ \mbox{ $x$ is absolutely continuous},
x(0)=x(1)=0, x'\in L^{2}[0,1]\big\}$$
is dense in $X$. Moreover, 
$T$ and $T^*$ are skew
 on $D$. 
\end{enumerate}
\end{theorem}
\begin{proof}
\ref{V:1}: $T$ and $T^*$ are maximal monotone because $T=V^{-1}$,
 and $T^*=(V^{-1})^*=(V^*)^{-1}$ and Example~\ref{ex:Volterra}\ref{readingbreak}.
 By Example~\ref{ex:Volterra}\ref{V:001},
  $T:L^{2}[0,1]\rightarrow L^{2}[0,1]$ has
\begin{align*}\dom T&=\{x\in L^{2}[0,1]:\ \mbox{ $x$ is absolutely continuous},
x(0)=0, x'\in L^{2}[0,1]\}\\
\dom T^*&=\{x\in L^{2}[0,1]:\ \mbox{ $x$ is absolutely continuous},
x(1)=0, x'\in L^{2}[0,1]\}
\\
Tx&=x',\;\forall x\in\dom T,\; T^*y=-y' \mbox{ and }\forall y\in\dom T^*.\end{align*}
  Note that by Fact~\ref{bypart},
\begin{equation}\label{quadratic1}
\scal{Tx}{x}=\int_{0}^{1}x'x=\frac{1}{2}x^2(1)-
\frac{1}{2}x^{2}(0)=\frac{1}{2}x(1)^2 \quad \forall x\in \dom T,
\end{equation}
\begin{equation}\label{quadratic2}
\scal{T^*x}{x}=\int_{0}^{1}-x'x=-(\frac{1}{2}x(1)^2
-\frac{1}{2}x(0)^2)=\frac{1}{2}x(0)^2\quad \forall x\in \dom T^*.
\end{equation}

\ref{V:3}: Letting $x(t)=t, y(t)=t^2$ we have
$$\scal{Tx}{x}=\int_{0}^{1}t=\tfrac{1}{2}, \quad
 \scal{x}{Ty}=\int_{0}^{1}2t^2=\tfrac{2}{3}\neq\tfrac{1}{3}=\int_{0}^{1}t^2=\scal{Tx}{y}\quad
\Rightarrow \scal{Tx}{x}\neq 0, \scal{Tx}{y}\neq \scal{x}{Ty}.$$

\ref{V:4}:
 By \ref{V:1}, $D=\dom T\cap\dom T^*$ is clearly a linear subspace.
For $x\in D$, $x(0)=x(1)=0$, from \eqref{quadratic1} and \eqref{quadratic2},
\begin{align*}\scal{Tx}{x}&=\tfrac{1}{2}x(1)^2=0,\quad
\scal{T^*x}{x}=\tfrac{1}{2}x(0)^2=0.
\end{align*}
Hence both $T$ and $T^*$ are skew  on $D$. The fact that
$D$ is dense in $L^{2}[0,1]$ follows from \cite[Theorem
6.111]{stromberg}.
\end{proof}

Our proof of \ref{V:3}, \ref{V:4} in the following theorem follows the ideas
of \cite[Example 13.4]{rudin}.

\begin{theorem}[The skew part of inverse Volterra operator]
\label{Eth:2}
Let $X=L^2[0,1]$, and $T$ be defined as in Theorem~\ref{Eth:1}. Let
$S:=\tfrac{T-T^*}{2}$.
\begin{enumerate}
\item \label{V:20} $Sx=x'\;(\forall x\in\dom S)$
 and  $\gra S=\{(Vx,x)\mid\;x\in e^{\bot}\}$,
where  $e\equiv 1\in L^2[0,1]$. In particular, \begin{align*} \dom S
&=\{x\in L^{2}[0,1]:\ \mbox{ $x$ is absolutely continuous},
x(0)=x(1)=0, x'\in L^{2}[0,1]\},\\
\ran S &= \{y\in L^{2}[0,1]: \scal{e}{y}=0\}=e^{\perp}.
\end{align*}

Moreover, $\dom S$ is dense, and
\begin{equation}\label{restriction}
S^{-1}=V|_{e^{\bot}}, \quad (-S)^{-1}=V^*|_{e^{\bot}},
\end{equation}
consequently, $S$ is skew, and neither $S$ nor $-S$ is maximal monotone.

\item \label{V:19}The adjoint of $S$ has $\gra S^*=\{(V^*x^*+le,x^*)
\mid x^*\in X,\;l\in\RR\}$. More precisely,
\begin{align*}
S^*x & =-x' \quad \forall x\in \dom S^*, \mbox{ with}\\
\dom S^* &=\{x\in L^{2}[0,1]:\ x \mbox{ is absolutely continuous
on $[0,1]$, $x'\in L^{2}[0,1]$}\},\\
\ran S^*& =L^{2}[0,1].
\end{align*}
 Neither $S^*$ nor $-S^*$
is monotone. Moreover, $S^{**}=S$.
\item\label{V:7} Let $T_1:\dom T_1\rightarrow X$ be defined by
\begin{align*}
T_1x=x',\quad \forall x\in\dom T_1:=\{x\in L^{2}[0,1]:\ \mbox{ $x$
is absolutely continuous}, x(0)=x(1), x'\in
L^{2}[0,1]\}.\end{align*} Then $T_{1}^*=-T_{1}$,
\begin{equation}\label{rangeoft1} \ran T_{1}=e^{\perp}.
\end{equation}
Hence $T_{1}$ is skew, and a maximal monotone extension of $S$; and
$-T_1$ is skew and a maximal monotone extension of $-S$.

\end{enumerate}
\end{theorem}

\begin{proof}
\ref{V:20}: By Theorem~\ref{Eth:1}\ref{V:4}, we directly get $\dom S$.
Now $(\forall x\in\dom S=\dom T\cap \dom T^*)\; Tx=x'$
and $T^*x=-x'$, so $Sx=x'$. Then  Example~\ref{ex:Volterra}\ref{V:0}
implies $ \gra S=\{(Vx,x)\mid\;x\in e^{\bot}\}$. Hence
\begin{equation}\label{sunday1}
\gra S^{-1}=\{(x, Vx): x\in e^{\bot}\}.
\end{equation}
Theorem~\ref{Eth:1}\ref{V:4} implies $\dom S$ is dense.
 Furthermore,
$\gra (-S)=\{(Vx,-x): x\in e^{\bot}\},$
so
$$\gra(-S)^{-1}=\{(x,-Vx): x\in e^{\bot}\}.$$
Since
$$V^*x(t)=\int_{t}^{1}x-0=\int_{t}^{1}x-\int_{0}^{1}x=-\int_{0}^{t}x=-Vx(t)
 \quad \forall t\in\left[0,1\right], \forall x\in e^{\bot}$$
we have $-Vx=V^*x,\forall x\in e^{\bot}$. Then
\begin{equation}\label{sunday2}
\gra(-S)^{-1}=\{(x,V^*x): x\in e^{\bot}\}.
\end{equation}
Hence, \eqref{sunday1} and \eqref{sunday2} together establish \eqref{restriction}.
As both $V, V^*$ are maximal monotone with full domain, we conclude that
$S^{-1}, (-S)^{-1}$ are not maximal monotone,
thus $S, -S$ are not maximal monotone.

\ref{V:19}: By \ref{V:20}, we have\begin{align*}
&(x,x^*)\in\gra S^*\Leftrightarrow
\langle -x, y\rangle+\langle x^*, Vy\rangle=0,\quad \forall y\in e^{\bot}\\
&\Leftrightarrow \langle -x+V^*x^*, y\rangle=0,\quad \forall
y\in e^{\bot} \Leftrightarrow x-V^*x^*\in\spand \{e\}.\end{align*}
Equivalently, $x=V^*x^*+ke$ for some $k\in \RR$. This means that $x$ is
absolutely continuous, $x^*=-x'\in L^{2}[0,1]$.
\sethlcolor{yellow}
On the other hand,
if $x$ is absolutely continuous and $x'\in L^{2}[0,1]$, observe that  
$$x(t)=\int_{t}^{1}-x'+x(1)e,$$
so that $x-V^*(-x')\in\spand\{e\}$ and $(x,-x')\in\gra S^*$. It
follows that
$$\dom S^*=\{x\in L^{2}[0,1]:\ x \mbox{ is absolutely continuous on
$[0,1]$}, x'\in L^{2}[0,1]\},$$ $$ \ran S^*=L^{2}[0,1], \quad \mbox{
and }$$
$$S^*x=-x', \ \forall x\in \dom S^*.$$
 Since
$$\scal{S^*x}{x}=-\int_{0}^{1}x'x=-\bigg(\frac{1}{2}x(1)^2-\frac{1}{2}x(0)^2\bigg),$$
we conclude that  neither $S^*$ nor $-S^*$ is monotone.

We proceed to show that $S^{**}=S$. Note that $\forall x\in \dom
S^*, z\in \dom S$,  we have $z(0)=z(1)=0$ and
$$\scal{S^*x}{z}=\int_{0}^{1}-x'z=-\bigg(x(1)z(1)-x(0)z(0)-\int_{0}^{1}xz'\bigg)
=\int_{0}^{1}xz'=
\scal{x}{Sz},$$ this implies that $S^{**}z=Sz, \ \forall z \in \dom
S$, i.e., $S^{**}|_{\dom S}=S$. Suppose now that $x\in \dom S^{**},
\phi=S^{**}x$. Put $\Phi = V$. 
Then $\forall z\in
\dom S^{*}$,
\begin{align*}
\scal{S^*z}{x} &=\int_{0}^{1}-z'x=\scal{z}{S^{**}x}\\
&=\scal{z}{\phi}=\int_{0}^{1}z\phi=[z(1)\Phi(1)-z(0)\Phi(0)]-\int_{0}^{1}\Phi
z'\\
&= z(1)\Phi(1)-\int_{0}^{1}\Phi z'.
\end{align*}
Using $z=e\in \dom S^*$ gives $\Phi(1)=0$. It follows that
$$\int_{0}^{1}[\Phi-x]z'=0, \quad \forall z\in \dom S^* \Rightarrow\Phi-x\in
 (\ran S^*)^{\perp},$$
then $\Phi =x$ since $\ran S^*=L^{2}[0,1]$. As $\Phi(1)=\Phi(0)=0$
and $\Phi$ is absolutely continuous, we have $x\in \dom S$. Since
$x\in \dom S^{**}$ was arbitrary, we conclude that $\dom
S^{**}\subseteq \dom S$. Hence $S^{**}=S$.
(Alternatively, 
$V$ is continuous
$\Rightarrow$
$V|_{e^\bot}$ has closed graph
$\Rightarrow$
$S^{-1}$ has closed graph
$\Rightarrow$
$S$ has closed graph
$\Rightarrow$
$\gra S = \gra S^{**}$ 
$\Rightarrow$
$S^{**} = S$.)

\ref{V:7}: To show \eqref{rangeoft1},  suppose that $x$ is
absolutely continuous and that $x(0)=x(1)$. Then
$$\int_{0}^{1}x'=x(1)-x(0)=0\quad\Rightarrow T_{1}x=x'\in e^{\perp}.$$
Conversely, if $x\in L^{2}[0,1]$ satisfies $\scal{e}{x}=0$, we
define $z=Vx$, then $z$ is absolutely
continuous, $z(0)=z(1)$, $T_{1}z=x$. Hence $\ran T_{1}=e^{\perp}$.

$T_1$ is skew, because for every $x\in\dom T_1$, we have 
\begin{align*}\langle T_1x,x\rangle=\int_{0}^{1}x'x=\tfrac{1}{2}x(1)^2-\tfrac{1}{2}x(0)^2=0.
\end{align*}

Moreover, $T^*_{1}=-T_1$: indeed, as $T_{1}$ is skew, by
Fact~\ref{EL:11}, $\gra(-T_{1})\subseteq \gra T_{1}^*$. To show that
$T_{1}^*=-T_{1}$, take $z\in \dom T^{*}_{1}, \phi
=T_{1}^{*}z$. Put $\Phi = V\phi$. We have $\forall y\in
\dom T_{1}$,
\begin{align}
\int_{0}^{1}y'z & =\scal{T_{1}y}{z}=\scal{T_{1}^*z}{y}
= \scal{\phi}{y}=\int_{0}^{1}y\phi =\int_{0}^{1}y\Phi' \label{foradjoint1}\\
&=[\Phi (1)y(1)-\Phi(0)y(0)]-\int_{0}^{1}\Phi y'. \label{foradjoint2}
\end{align}
Using $y=e\in \dom T_{1}$ gives $\Phi(1)-\Phi(0)=0$, from which
$\Phi(1)=\Phi(0)=0$. It follows from
\eqref{foradjoint1}--\eqref{foradjoint2} that
$\int_0^{1}y'(z+\Phi)=0\ \forall y\in\dom T_{1}$. Since $\ran
T_{1}=e^{\perp}$, $z+\Phi\in\spand\{e\}$, say $z+\Phi=ke$ for some
constant $k\in \RR$. Then $z$ is absolutely continuous, $z(0)=z(1)$
since $\Phi(0)=\Phi(1)=0$, and $T^*_{1}z=\phi=\Phi'=-z'$. This
implies that $\dom T_{1}^*\subseteq \dom T_{1}$.  Then by Fact~\ref{EL:11},
$T_{1}^*=-T_{1}$. It remains to apply Proposition~\ref{skewmax}.
\end{proof}

\begin{fact}\label{FT:1} Let $A:X\To X$ be a multifunction. Then
$(-A)^{-1}=A^{-1}\circ (-\Id).$ If $A$ is a linear relation, then
$$(-A)^{-1}=-A^{-1}.$$
\end{fact}
\begin{proof}
This follows from the set-valued inverse definition. Indeed,
$x\in (-A)^{-1}(x^*)\Leftrightarrow (x,x^*)\in \gra (-A) \Leftrightarrow (x, -x^*)\in \gra A
\Leftrightarrow x\in A^{-1}(-x^*).$
When $A$ is a linear relation, $x\in (-A)^{-1}(x^*)\Leftrightarrow (x,-x^*)\in \gra A \Leftrightarrow
(-x,x^*)\in \gra A \Leftrightarrow -x\in A^{-1}x^*\Leftrightarrow x\in -A^{-1}(x^*).$
\end{proof}

\begin{theorem}[The inverse of the skew part of Volterra operator]\label{Eth:3}
Let $X=L^2[0,1]$, and $V$ be the Volterra integration operator,
and $V_{\circ}:L^{2}[0,1]\rightarrow L^{2}[0,1]$ be
given by
$$V_{\circ}=\frac{V-V^*}{2}.$$
 Define $T_2:\dom T_{2}\rightarrow
L^{2}[0,1]$ by $T_{2}=V_{\circ}^{-1}$. Then
\begin{enumerate} \item $T_{2}x=x', \quad \forall x\in \dom T_{2}$
where
\begin{equation}\label{ranofV0}
\dom T_{2}=\{x\in L^{2}[0,1]: \mbox{ $x$ is absolutely continuous
on $[0,1]$, $x'\in L^{2}[0,1], x(0)=-x(1)$}\}.
\end{equation}
 \item
$T_{2}^*=-T_{2}$, and both $T_{2}, -T_{2}$ are maximal monotone 
and skew.
\end{enumerate}
\end{theorem}
\begin{proof}
(i) Since
$$V_{\circ}x(t)=\tfrac{1}{2}\bigg(\int_{0}^{t}x-\int_{t}^{1}x\bigg),$$
$V_{\circ}$ is a one-to-one map. Then
$$V_{\circ}^{-1}\bigg(\frac{1}{2}(\int_{0}^{t}x-\int_{t}^{1}x)\bigg)=x(t)
=\bigg(\frac{1}{2}(\int_{0}^{t}x-\int_{t}^{1}x)\bigg)',$$
which implies $T_{2}x=V_{\circ}^{-1}x=x'$ for $x\in \ran V_{\circ}$.
As $\dom T_{2}=\ran V_{\circ}$, by Example~\ref{ex:Volterra}\ref{FI:4},
 $\ran V_{\circ}$ can be
written as \eqref{ranofV0}.

(ii) Since $\dom V=\dom V^*=L^{2}[0,1]$, $V_{\circ}$ is skew on
$L^{2}[0,1]$, so maximal monotone. Then $T_{2}=V_{\circ}^{-1}$ is
maximal monotone.

Since $V_{\circ}$ is skew and $\dom V_{\circ}=L^{2}[0,1]$, we have
$V_{\circ}^*=-V_{\circ}$, by Fact~\ref{FT:1},
$$T_{2}^*=(V_{\circ}^{-1})^{*}=(V_{\circ}^*)^{-1}=(-V_{\circ})^{-1}=-V_{\circ}^{-1}=-T_{2}.$$
 By Proposition~\ref{skewmax}, both $T_{2}$ and $-T_{2}$ are maximal
monotone and skew.
\end{proof}
\begin{remark} Note that while $V_{\circ}$ is continuous on $L^{2}[0,1]$, the operator $S$ given in
Example~\ref{1EL:0} is discontinuous.
\end{remark}
Combining Theorem~\ref{Eth:1}, Theorem~\ref{Eth:2} and Theorem~\ref{Eth:3}, we can
summarize the nice relationships among the differentiation
operators encountered in this section. 
\begin{corollary} The domain of the skew operator $S$ is dense in $L^{2}[0,1]$.
Neither $S$ nor $-S$ is maximal monotone. Neither $S^*$ nor $-S^*$ is monotone.

The linear operators $S, T, T_{1}, T_{2}$ satisfy:
$$\gra S\subsetneqq \gra T\subsetneqq \gra(-S^*),$$
$$\gra S\subsetneqq \gra T_{1}\subsetneqq \gra(-S^*),$$
$$\gra S\subsetneqq \gra T_{2}\subsetneqq \gra(-S^*).$$
While $S$ is skew, $T, T_{1}, T_{2}$ are maximal monotone and
$T_{1}, T_{2}$ are skew. Also,
$$\gra(-S)\subsetneqq \gra(T^*)\subsetneqq \gra S^*,$$
$$\gra(-S)\subsetneqq \gra(-T_{1})\subsetneqq \gra S^*,$$
$$\gra(-S)\subsetneqq \gra(-T_{2})\subsetneqq \gra S^*.$$
While $-S$ is skew, $T^*, -T_{1}, -T_{2}$ are maximal monotone
and $-T_{1}, -T_{2}$ are skew.
\end{corollary}
\begin{remark}
(i).
Note that while $T_{1},T_{2}$ are maximal monotone, $-T_{1},-T_{2}$ are also maximal monotone.
This is in stark contrast with the maximal monotone skew operator given
 in Proposition~\ref{EL:0L} and Proposition~\ref{EL:10} such that its negative is not maximal monotone.

(ii). Even though the skew operator $S$ in Theorem~\ref{Eth:2}  has $\dom S$ dense in $L^{2}[0,1]$,
 it still admits two distinct maximal
monotone and skew extensions $T_{1},T_{2}$.
\end{remark}

\subsection{Consequences on sum of maximal monotone operators and Fitzpatrick functions of a sum}

\begin{example}[$T+T^*$ fails to be maximal monotone]
Let $T$ be defined as in Theorem~\ref{Eth:1}.
Now $\forall x\in \dom T\cap \dom T^*$, we have
$$Tx+T^*x=x'-x'=0.$$
Thus $T+T^*$ has a proper monotone extension
 from $\dom T\cap \dom T^*\subsetneqq X$ to the $0$ map on $X$.
Consequently, $T+T^*$ is not maximal monotone.
 Note that $\dom T\cap \dom T^*$ is dense in $X$
and that $\dom T-\dom T^*$ is a dense subspace of $X$.
This supplies a simpler example for showing that the constraint
qualification in the sum problem of maximal
 monotone operators can not be substantially weakened, see \cite[Example
7.4]{PheSim}. Similarly, by Theorems~\ref{Eth:2} and \ref{Eth:3},
$T_{i}^*=-T_{i}$, we conclude that
$T_{i}+T_{i}^*=0$ on $\dom T_{i}$, a dense subset  of $L^{2}[0,1]$; thus,
$T_{i}+T_{i}^*$ fails to be maximal monotone while both $T_{i}, T_{i}^*$ 
are maximal monotone.
\end{example}

To study Fitzpatrick functions of sums of maximal monotone operators, we need:
\begin{lemma}\label{Vle:1}
Let $V$ be the Volterra integration operator. Then 
\begin{align*}
q_{V_+}^*(z)=\iota_{\spand\{e\}}(z)+\langle z, e\rangle^2, \quad \forall z\in X.
\end{align*}
\end{lemma}
\begin{proof}
Let $z\in X$.
By Example~\ref{ex:Volterra}\ref{V:010} and Fact~\ref{f1:Fitz}, we have
\begin{align*}
q_{V_+}^*(z)=\infty, \quad \text{if $z\notin \spand\{e\}$.}
\end{align*}
Now suppose that $z=le$ for some $l\in\RR$. By Example~\ref{ex:Volterra}\ref{V:010},
\begin{align*}
q_{V_+}^*(z)&=\sup_{x\in X}\{\langle x,z\rangle-q_{V_+}(x)\}
=\sup_{x\in X}\{\langle x,le\rangle-\tfrac{1}{4}\langle x,e\rangle^2\}\\
&=l^2=\langle le, e\rangle^2=\langle z, e\rangle^2.
\end{align*}
Hence $q_{V_+}^*(z)=\iota_{\spand\{e\}}(z)+\langle z, e\rangle^2$.
\end{proof}
\begin{lemma}Let $T$ be defined as in Theorem~\ref{Eth:1}. We have
\begin{align}
F_T(x,y^*)&=F_V(y^*,x)=\iota_{\spand\{e\}}(x+V^*y^*)+\tfrac{1}{2}\langle x+V^*y^*, e\rangle^2,\nonumber\\
F_{T^*}(x,y^*)&=F_{V^*}(y^*,x)=\iota_{\spand\{e\}}(x+Vy^*)+\tfrac{1}{2}\langle x+Vy^*, e\rangle^2
,\; \forall (x,y^*)\in X\times X.\label{FI:9}\end{align}
\end{lemma}
\begin{proof}
Apply Fact~\ref{f:Fitz}, Fact~\ref{f1:Fitz} and Lemma~\ref{Vle:1}.
\end{proof}

\begin{remark}
Theorem~\ref{FI:7} below gives another example showing that
$F_{T+T^*}\neq F_{T}\Box_{2} F_{T^*}$ while $T,T^*$ are maximal monotone, and $\dom T-\dom T^*$ is a
dense subspace in $L^{2}[0,1]$. Moreover, $\ran(T_{+}+(T^*)_{+})=\{0\}$. This again shows that the assumption that $\dom A-\dom B$ is closed in Fact~\ref{fitzpatrickfunctionsum}(ii) can not be
weakened substantially, and that Fact~\ref{fitzpatrickfunctionsum}(i) fails for discontinuous linear monotone
operators.
\end{remark}

\begin{theorem}\label{FI:7}
Let $T$ be defined as in Theorem~\ref{Eth:1},
and set \begin{align*}H:=\{x\in L^{2}[0,1]:\ \mbox{ $x$ is absolutely continuous, and}\;
x'\in L^{2}[0,1]\}.\end{align*}
Then 
\begin{align}
 F_{T+T^*}(x,x^*)&=\iota_{X\times\{0\}}(x,x^*),\quad \forall(x,x^*)\in X\times X\nonumber\\
F_T\Box_2 F_{T^*}(x,x^*)&=
\begin{cases}
\tfrac{1}{2}\left[x(1)^2
+x(0)^2\right],\;&\text{if}\; (x,x^*)\in H\times\{0\};\\
\infty,\;&\text{otherwise.}\end{cases}
\label{sump:01}\end{align}
Consequently,
$F_T\Box_2F_{T^*}\neq F_{T+T^*}.$
\end{theorem}
\begin{proof}
 By Theorem~\ref{Eth:1}\ref{V:1} 
\sethlcolor{yellow}and Example~\ref{ex:Volterra}\ref{V:0}, \begin{align}(T+T^*)y=0,
\forall y\in\dom T\cap \dom T^*=\{Vx~\mid x\in e^{\bot}\},\label{FI:1}\end{align} where $e\equiv 1\in L^2[0,1]$.
Let $(x,x^*)\in X\times X$. 
Using Theorem~\ref{Eth:1}\ref{V:1}, we see that 
\begin{equation}
F_{T+T^*}(x,x^*)=\sup_{y\in\dom T\cap\dom T^*}\langle x^*,y\rangle
=\sup_{y\in X}\langle x^*,y\rangle 
=\iota_{\{0\}}(x^*)=\iota_{X\times\{0\}}(x,x^*).\label{FI:2}
\end{equation}
By Fact~\ref{psum}, we have
\begin{align}
 \big(F_T\Box_2F_{T^*}\big)(x,x^*)=\infty,\quad \forall \ x^*\neq 0.\label{sump:1}\end{align}
When $x^*=0$, by \eqref{FI:9},
\begin{align}
& \big(F_T\Box_2F_{T^*}\big)(x,0)=\inf_{y^*\in X}\{
F_T(x,y^*)+F_{T^*}(x,-y^*)\}\label{FI:3}\\
&=\inf_{y^*\in X}\{\iota_{\spand\{e\}}(x+V^*y^*)+\tfrac{1}{2}\langle x+V^*y^*,e\rangle^2
+\iota_{\spand\{e\}}(x-Vy^*)+\tfrac{1}{2}\langle x-Vy^*,e\rangle^2\}\nonumber.
\end{align}
Observe that
\begin{align*} &x+V^*y^*\in \spand\{e\}, x-Vy^*\in \spand\{e\}\\
&\Leftrightarrow x-Vy^*+Vy^*+V^*y^*\in \spand\{e\}, x-Vy^*\in \spand\{e\}\\
&\Leftrightarrow x-Vy^*\in \spand\{e\},
\quad(\text{by Example~\ref{ex:Volterra}\ref{V:010}})\\
&\Leftrightarrow x\in V y^*+\spand\{e\}\Leftrightarrow \text{$x$ is absolutely continuous and $y^*=x'$}.
\end{align*}
Therefore, $(F_T\Box_2F_{T^*})(x,0)=\infty$ if $x\notin H$.
For $x\in H$,
using \eqref{FI:3} 
\sethlcolor{lime}
{and the fact that $x-Vx' = x(0)e$ and $x+V^*x' = x(1)e$}, 
we obtain
\begin{align*}
 &\big(F_T\Box_2 F_{T^*}\big)(x,0)=\tfrac{1}{2}\langle x+V^*x',e\rangle^2
+\tfrac{1}{2}\langle x-Vx',e\rangle^2\\
&=\tfrac{1}{2}x(1)^2
+\tfrac{1}{2}x(0)^2=\tfrac{1}{2}\left[x(1)^2
+x(0)^2\right].\end{align*}
Thus, \eqref{sump:01} holds.
Consequently,
 $F_T\Box_2F_{T^*}\neq F_{T+T^*}.$
\end{proof}

Finally, we remark that the examples given in 
Sections~\ref{s:skew} and \ref{s:vol}
 have important consequences  on \emph{decompositions of monotone operator},
  namely Borwein-Wiersman decomposition and Asplund decomposition \cite{bwiersma}.
This will be addressed in the forthcoming paper \cite{BWY6}.

\section*{Acknowledgment}
Heinz Bauschke was partially supported by the Natural Sciences and
Engineering Research Council of Canada and
by the Canada Research Chair Program.
Xianfu Wang was partially supported by the Natural
Sciences and Engineering Research Council of Canada.

\end{document}